\documentclass[12pt,leqno]{article}
\usepackage{amssymb, amscd, amsmath, amsthm}
\usepackage{dsfont}
\allowdisplaybreaks

\def\beq{\begin{equation}}
\def\eeq{\end{equation}}

\theoremstyle{definition}
\newtheorem{definition}{Definition}

\newtheorem{observation}{Observation}
\newtheorem*{rema}{Remark}

\theoremstyle{plain}
\newtheorem{theorem}{Theorem}

\newtheorem{lemma}{Lemma}

\newtheorem{corollary}{Corollary}
\newtheorem{proposition}{Proposition}

\numberwithin{equation}{section}
\numberwithin{proposition}{section}
\numberwithin{observation}{section}
\numberwithin{definition}{section}
\numberwithin{theorem}{section}
\numberwithin{problem}{section}
\numberwithin{example}{section}
\numberwithin{claim}{section}
\numberwithin{fact}{section}
\numberwithin{lemma}{section}
\numberwithin{conjecture}{section}
\numberwithin{corollary}{section}

\begin{document}

\title{Graphs without rainbow triangles}

\author{by P. Frankl
\\
R\'enyi Institute, Budapest, Hungary\thanks{The research was supported by the National Research, Development and Innovation Office NKFIH, Hungary,
grant number K132696.}
}

\date{}
\maketitle

\begin{abstract}
Let $\mathcal G_1, \mathcal G_2, \ldots, \mathcal G_t$ be graphs on the same $n$ vertices.
Assuming that there is no way to choose three edges from distinct $\mathcal G_i$ that form a triangle we determine the maximum of $\bigl|\mathcal G_1\bigr| + \ldots + \bigl|\mathcal G_t\bigr|$.
Under the same conditions and $t = 3$ we conjecture that $\bigl|\mathcal G_1\bigr| \bigl|\mathcal G_2\bigr| \bigl|\mathcal G_3\bigr| \leq \left(n^2/4\right)^3$ holds.
This inequality is proved under some additional conditions.
\end{abstract}

\section{Introduction}
\label{sec:1}

Let $(V, \mathcal E)$ be a graph with vertex-set $V$ and edge-set $\mathcal E \subset {V\choose 2}$.
When it causes no confusion we shall omit $V$.
Let us use the notation $\mathcal E(x) = \{y \in V : (x, y) \in \mathcal E\}$ (the neighbourhood of $x$) and
$\mathcal E(\overline x) = \{E \in \mathcal E : x \notin E\}$, the subgraph spanned by $V\setminus \{x\}$.
Note the obvious relation $|\mathcal E| = |\mathcal E(x)| + |\mathcal E(\overline x)|$.

A \emph{triangle} is the complete graph on three vertices, $\left(T, {T\choose 2}\right)$: $T \in {V\choose 3}$.
A \emph{mathching} is a collection $\mathcal M = \{E_1, \ldots, E_\ell\}$ of pairwise disjoint edges and $\ell$ is its \emph{size}.
The maximum size of a matching in $\mathcal E$ is denoted by $\nu(\mathcal E)$, it is called the \emph{matching number}.

For a fixed graph $\mathcal G$, $(V, \mathcal E)$ is called \emph{$\mathcal G$-free} if it contains no subgraph isomorphic to~$\mathcal G$.

\begin{definition}
\label{def:1.1}
For a positive integer $n$ and a fixed graph $\mathcal G$ let $m(n, \mathcal G)$ denote the maximum of $|\mathcal E|$ where $(V, \mathcal E)$ is $\mathcal G$-free and $|V| = n$.
\end{definition}

The first such result was due to Mantel (1907), cf.\ \cite{B} or \cite{L}.
It states that
\beq
\label{eq:1.1}
m(n, \text{ triangle}) = \lfloor n^2/4\rfloor.
\eeq

This simple result went unnoticed and the now burgeoning field of \emph{extremal graph theory} came to existence only after Tur\'an \cite{T} determined $\text{\rm ex}(n, K_r)$ where $K_r$ is the complete graph on $r$ vertices.

We'll need Mantel's theorem in the following stronger form.

\setcounter{proposition}{1}
\begin{proposition}
\label{prop:1.2}
Let $n, \ell$ be positive integers, $n \geq 2\ell$ and $(V, \mathcal E)$ a triangle-free graph with $|V| = n$, $\nu(\mathcal E) = \ell$.
Then
\beq
\label{eq:1.2}
|\mathcal E| \leq \ell(n - \ell).
\eeq

Moreover there exists a partition $V = X \sqcup Y \sqcup Z$, $X = \{x_1, \ldots, x_\ell\}$, $Y = \{y_1, \ldots y_\ell\}$ with the following properties

\smallskip
\noindent
{\rm (i)}\quad\   $(x_i, y_i) \in \mathcal E$, \ $1 \leq i \leq \ell$.

\smallskip
\noindent
{\rm (ii)}\quad For every $z \in Z$, \ $\mathcal E(z) \subset X$.
\end{proposition}

Note that \eqref{eq:1.2} follows from (i) and (ii).
Indeed, if $w, x_i, y_i$ are three distinct elements of $X \cup Y$, then the absence of triangles implies that $w$ is connected to at most one of the two vertices $x_i$ and $y_i$.
Consequently, the degree of (the arbitrary vertex) $w$ in $X \cup Y$ is at most $\ell$.
Thus $\mathcal E$ restricted to $X \cup Y$ has at most $2\ell \times \ell/ 2 = \ell^2$ edges.

In view of (ii) the number of edges adjacent to $Z$ is at most $|Z|\ell = (n - 2\ell)\ell$.
Thus $|\mathcal E| \leq \ell^2 + (n - 2\ell)\ell = (n - \ell)\ell$.

\begin{proof}[Proof of {\rm (i)} and {\rm (ii)}]
Let $(v_i, w_i)$, $1 \leq i \leq \ell$ be a matching of size $\ell$ in $\mathcal E$ and set $W = \{v_1, w_1, \ldots, v_\ell, w_\ell\}$, $Z = V \setminus W$.
The maximality of the matching implies that $\mathcal E \cap {Z\choose 2} = \emptyset$.
If there is an edge $(z, x_i) \in \mathcal E$ with $z \in Z$ and $x_i \in (v_i, w_i)$, then we put $x_i$ into $X$ and let the other vertex of $(v_i, w_i)$ be $y_i$.
The important observation is that $(z, y_i) \notin \mathcal E$ (it would finish a triangle) and $(z', y_i) \notin \mathcal E$ for $z'\in Z$, $z'\neq z$ as replacing $(x_i, y_i)$ by $(x_i, z)$ and $(y_i, z')$ would produce a larger matching.
If there is no edge connecting $Z$ and $(v_i, w_i)$, then we put arbitrarily one vertex into $X$ and the other into $Y$.
It should be clear that $V = X \sqcup Y \sqcup Z$ is a partition with properties (i) and (ii).
\end{proof}

Recently so-called \emph{rainbow} structures have received quite some attention.
Confer the excellent survey article by Fujita, Magnant and Ozeki \cite{FMO}.

\setcounter{definition}{2}
\begin{definition}
\label{def:1.3}
Let $\mathcal E_i \subset {V\choose 2}$, $1 \leq i \leq t$ and let $\mathcal G$ be a fixed graph, $s = |\mathcal G|$.
If for some choice of $1 \leq i_1 < \ldots < i_s \leq t$ and edges $E_{i_j} \in \mathcal E_{i_j}$ the graph $\bigl\{E_{i_1},\ldots, E_{i_s}\bigr\}$ is isomorphic to $\mathcal G$, then it is called a rainbow copy of~$\mathcal G$.
If no such copy exists, $\mathcal E_1,\ldots, \mathcal E_t$ are said to be \emph{rainbow $\mathcal G$-free}, or \emph{RBG}-free for short.
If $s > t$, then no rainbow copy of $\mathcal G$ exists.
\end{definition}

Setting $T$ for the triangle our first result is the following

\setcounter{theorem}{3}
\begin{theorem}\hspace*{-5pt}\footnote{The referee pointed out that Theorem \ref{th:1.4} was proved in a more general form in \cite{KSSW}.
Our proof is different.}
\label{th:1.4}
Let $t \geq 3$ and suppose that $\mathcal G_1, \ldots, \mathcal G_t \subset {V\choose 2}$, $|V| = n$ are RBT-free.
Then {\rm (i)} or {\rm (ii)} hold

\noindent
{\rm (i)}\quad\ $t = 3$ and
$$
|\mathcal G_1| + |\mathcal G_2| + |\mathcal G_3| \leq n(n - 1),
$$
{\rm (ii)}\quad $t \geq 4$ and
\beq
\label{eq:1.3}
\bigl|\mathcal G_1\bigr| + \ldots + \bigl|\mathcal G_t\bigr| \leq t \bigl\lfloor n^2/4\bigr\rfloor.
\eeq
\end{theorem}

\begin{rema}
Setting $\mathcal G_1 = \ldots = \mathcal G_t = \mathcal E$, \eqref{eq:1.3} implies Mantel's theorem.
Letting $\mathcal G_1 = \mathcal G_2$ be the complete graph and $\mathcal G_3$ the empty graph (on $n$ vertices) shows that (i) is best possible.
For $n \geq 5$ this is the essentially unique way to achieve equality.
For $t \geq 4$ letting $\mathcal G_1, \ldots, \mathcal G_t$ be the same complete bipartite graph with partite sets of size $\bigl\lfloor n/2\bigr\rfloor$ and $\bigl\lfloor(n + 1)/2\bigr\rfloor$ provides the essentially unique example for equality.
\end{rema}

We should note that knowing \eqref{eq:1.3} for a certain value of $t$ implies \eqref{eq:1.3} for $t + 1$.
Indeed, suppose that \eqref{eq:1.3} holds for $t$, $\mathcal G_1, \ldots, \mathcal G_{t + 1}$ are RBT-free and by symmetry
$\bigl|\mathcal G_1\bigr| \geq \ldots \geq \bigl|\mathcal G_t\bigr| \geq \bigl|\mathcal G_{t + 1}\bigr|$.
From \eqref{eq:1.3}, $\bigl|\mathcal G_t\bigr| \leq n^2/4$ and thereby $\bigl|\mathcal G_{t + 1}\bigr| \leq \bigl\lfloor n^2/4\bigr\rfloor$ follow.
Thus \eqref{eq:1.3} holds for $t + 1$ as well.

Consequently we only need to prove \eqref{eq:1.3} for $t = 4$.
As we will show later, to prove (i) and (ii) it is sufficient to consider \emph{nested families} of graphs, i.e., 
$\mathcal G_1, \ldots, \mathcal G_t$ satisfying $\mathcal G_1 \subset \ldots \subset \mathcal G_t$.
For such graphs we prove a stronger inequality in the case $t = 3$.

\begin{theorem}
\label{th:1.5}
Suppose that $\mathcal G_1 \subset \mathcal G_2 \subset \mathcal G_3 \subset {V\choose 2}$, $|V| = n$ and $\mathcal G_1, \mathcal G_2, \mathcal G_3$ are RBT-free. Then
\beq
\label{eq:1.4}
\bigl|\mathcal G_1\bigr|\bigl|\mathcal G_2\bigr|\bigl|\mathcal G_3\bigr| \leq \bigl\lfloor n^2/4\bigr\rfloor^3.
\eeq
\end{theorem}

\begin{rema}
We are going to prove \eqref{eq:1.4} under the slightly weaker condition $\mathcal G_1 \subset \mathcal G_2 \cap \mathcal G_3$, i.e., without requiring $\mathcal G_2 \subset \mathcal G_3$.
The proof is short and elementary.
\end{rema}

\section{The proof of Theorem \ref{th:1.4}}
\label{sec:2}

Let $(V, \mathcal E)$ be a graph.
A subset $S \subset V$ is called a \emph{transversal} or an \emph{edge-cover} if for every $E \in \mathcal E$, $S \cap E \neq \emptyset$.
The inequality $|S| \geq \nu(\mathcal E)$ should be obvious.
K\"onig \cite{K} proved that for bipartite graphs equality holds.
Let us state a simple consequence of it.

\begin{proposition}
\label{prop:2.1}
Suppose that $\mathcal B$ is a bipartite graph with partite sets $X$ and $Y$, $|X| = |Y| =:q$ and $\nu(\mathcal B) < q$.
Then
\beq
\label{eq:2.1}
|\mathcal B| \leq (q - 1) q
\eeq
with equality iff $\mathcal B$ is a complete bipartite graph with partite sets of size $q - 1$ and $q$ (plus an isolated vertex).
\end{proposition}

\setcounter{corollary}{1}
\begin{corollary}
\label{cor:2.2}
Suppose that $\mathcal G_1, \mathcal G_2, \mathcal G_3$ are RBT-free and let $\mathcal T = \bigl\{T_1, T_2, T_3\bigr\}$ be a triangle, i.e., $\mathcal T = {Z\choose 2}$ for some $3$-set $Z$.
Then
\beq
\label{eq:2.2}
\sum_{1 \leq i \leq 3} \bigl|\mathcal G_i \cap \mathcal T\bigr| \leq 6.
\eeq
\end{corollary}

\begin{proof}
Construct a bipartite graph $\mathcal F$ with partite sets $X = \{1, 2, 3\}$ and $Y = \bigl\{T_1, T_2, T_3\bigr\}$ by making $(i, T_j)$ an edge of $\mathcal F$ iff $T_j \in \mathcal G_i$.

A perfect matching in $\mathcal F$ corresponds to a rainbow triangle.
Thus \eqref{eq:2.2} follows from \eqref{eq:2.1}.
\end{proof}

\begin{proof}[The proof of Theorem \ref{th:1.4} {\rm (i)}]
Adding \eqref{eq:2.2} for all ${n \choose 3}$ subsets $Z \in {V\choose 3}$ yields
$$
\sum_{1 \leq i\leq 3} \sum_{Z \in {V\choose 3}} \bigl|\mathcal G_i \cap \mathcal T\bigr| \leq 6{n \choose 3} = n(n - 1)(n - 2).
$$

Since each incidence $T_j \in \mathcal G_i$ is counted exactly $n - 2$ times (once for every $Z$, $T_j \subset Z\in {V \choose 3}$),
the LHS equals $(n - 2)\bigl(\bigl|\mathcal G_1\bigr| + \bigl|\mathcal G_2\bigr| + \bigl|\mathcal G_3\bigr|\bigr)$ and (i) follows.

\hfill $\square$

\smallskip
In case of equality, equality must hold in \eqref{eq:2.2} for \emph{all} ${n\choose 3}$ choices of~$Z$.
Using the uniqueness part of Proposition \ref{prop:2.1} and ``continuity'' we infer that either for \emph{each} triangle $T$ there are two edges contained in all three graphs $\mathcal G_1, \mathcal G_2, \mathcal G_3$.
Or \emph{each} triangle is contained in exactly two of the graphs $\mathcal G_1, \mathcal G_2, \mathcal G_3$.

It is straightforward to check that for $n = 5$ (and thus for $n \geq 5$) the first case is impossible.
In the second case it easily follows that the triangles must always be in the same two graphs.
Consequently $\mathcal G_i = {V\choose 2}$ holds for two of the graphs and the third is empty.
\end{proof}

Let us next deal with the case $t \geq 4$.
We need two simple statements.
Let us use the notation $d_{\mathcal G}(x, Y)$ to denote the number of edges of the form $(x, y) \in \mathcal G$ with $y \in Y$.

\setcounter{lemma}{2}

\begin{lemma}
\label{lem:2.3}
Suppose that $\mathcal G \subset {V\choose 2}$ is triangle-free, $\bigl\{E_1, \ldots E_\ell\bigr\} \subset \mathcal G$ is a matching with $W = E_1 \cup \ldots \cup E_\ell$.
Then {\rm (i)} and {\rm (ii)} hold.

\smallskip
\noindent
{\rm (i)}\quad\ $d_{\mathcal G}(x, W) \leq \ell \quad \text{ for all } \ x \in V\setminus W$.

\smallskip
\noindent
{\rm (ii)}\quad If $E_1, \ldots, E_\ell$ form a maximal matching, then
$$
d_{\mathcal G}\bigl(x, V/\{x\}\bigr) \leq \ell \quad \text{ for all } \ x \in V\setminus W.
$$
\end{lemma}

\begin{proof}
To prove (i) just note that $x$ being adjacent to both endvertices of $E_j$ would force a triangle.

To prove (ii) notice further that $x$ being adjacent to a vertex $y \notin W$ would force a larger matching.
\end{proof}

\begin{lemma}
\label{lem:2.4}
Suppose that $\mathcal B, \mathcal C, \mathcal D \subset {V \choose 2}$ are RBT-free and let $\bigl\{E_1, \ldots, E_\ell\bigr\} \subset \mathcal B$ be a matching.
Set $W = E_1 \cup \ldots \cup E_\ell$ and fix $x \in V/W$.
Then
\beq
\label{eq:2.3}
d_{\mathcal C}(x, W) + d_{\mathcal D}(x, W) \leq 2\ell.
\eeq
\end{lemma}
\begin{proof}
For each edge $E_j$, $1 \leq j \leq \ell$, $E_j \in \mathcal B$ because $\{x\} \cup E_j$ is not spanning a rainbow triangle, $d_{\mathcal C}(x, E_j) + d_{\mathcal D} (x, E_j) \leq 2$.

Summing this inequality for $1 \leq j \leq \ell$ yields \eqref{eq:2.3}.
\end{proof}

\setcounter{definition}{4}
\begin{definition}
\label{def:2.5}
Let us call the graph $(V, \mathcal G)$ \emph{nearly matchable} if it possesses a matching of size $\ell$ with $2\ell \geq n - 2$.
\end{definition}

\setcounter{proposition}{5}
\begin{proposition}
\label{prop:2.6}
Suppose that $\mathcal B, \mathcal C, \mathcal D \subset {V\choose 2}$ are RBT-free and $\mathcal B$ is nearly matchable.
Then
\beq
\label{eq:2.4}
|\mathcal C| + |\mathcal D| \leq 2 \bigl\lfloor n^2/4\bigr\rfloor.
\eeq
\end{proposition}

\begin{proof}
\eqref{eq:2.4} holds for $n = 1$ and $2$.
Let us prove it for $n \geq 3$ by applying induction on $n$.
Let $\bigl\{E_1, \ldots, E_\ell\bigr\} \subset \mathcal B$ be a matching with $2\ell = n - 1$ or $n - 2$.
Fix $y \in V \setminus \bigl\{E_1 \cup \ldots \cup E_\ell\bigr\}$.
Note that omitting the vertex $y$, the remaining graph $\mathcal B(\overline{y})$ is nearly matchable.
By the induction hypothesis
\beq
\label{eq:2.5}
|\mathcal C(\overline{y})| + |\mathcal D(\overline y)| \leq 2 \left\lfloor (n - 1)^2/4\right\rfloor.
\eeq
Set $\mathcal G(y) = \{x: \{x, y\} \in \mathcal G\}$.
Note that in the case $2\ell = n - 1$, $V \setminus \{y\} = E_1 \cup \ldots \cup E_\ell$.
Thus \eqref{eq:2.3} implies $|\mathcal C(y)| + |\mathcal D(y)| \leq 2\ell = n - 1$.
Adding this to \eqref{eq:2.5}, $|\mathcal C| + |\mathcal D| \leq 2\ell^2 + 2\ell = 2(\ell^2 + \ell) = 2 \left\lfloor (2\ell + 1)^2/4\right\rfloor$ follows.

In the case $2\ell = n - 2$, i.e., $2\ell + 1 = n - 1$ there is one more vertex in $V \setminus \bigl\{E_1 \cup \ldots \cup E_\ell\bigr\}$.
Thus \eqref{eq:2.3} implies $|\mathcal C(y)| + |\mathcal D(y)| \leq 2\ell + 2$.
Adding this to \eqref{eq:2.5} yields
$
|\mathcal C| + |\mathcal D| \leq 2 \left\lfloor (2\ell + 1)^2/4\right\rfloor + 2(\ell + 1) = 2(\ell + 1)^2 = 2\left\lfloor n^2/4\right\rfloor$, as desired.
\end{proof}

We shall deduce \eqref{eq:1.3} from the following statement.

\setcounter{theorem}{6}
\begin{theorem}
\label{th:2.7}
Suppose that $\mathcal B, \mathcal C, \mathcal D \subset {V\choose 2}$ are RBT-free and $\mathcal B$ is triangle-free.
Then
\beq
\label{eq:2.6}
2|\mathcal B| + |\mathcal C| + |\mathcal D| \leq 4 \left\lfloor n^2/4\right\rfloor.
\eeq
\end{theorem}

\begin{proof}
The statement is trivial for $n = 1,2$.
Let us apply induction to prove \eqref{eq:2.6}.
By Mantel's theorem we have:
\beq
\label{eq:2.7}
|\mathcal B| \leq \left\lfloor n^2/4\right\rfloor.
\eeq
Now if $\mathcal B$ is nearly matchable, then \eqref{eq:2.6} follows from \eqref{eq:2.7} and \eqref{eq:2.4}.
Thus we may assume that $\mathcal B$ is not nearly matchable, $\bigl\{E_1, \ldots, E_\ell\bigr\} \subset \mathcal B$ is a maximal matching, $W = E_1 \cup \ldots \cup E_\ell$, $2\ell \leq n - 2$.
Fix $x \in V \setminus W$.
In view of Lemma \ref{lem:2.3} (ii) $2d_{\mathcal B}(x) \leq 2\ell$.

On the other hand \eqref{eq:2.3} shows that counting with multiplicity there are at least $2\ell$ edges of the form $(x, w)$, $w \in W$ missing from $\mathcal C$ and $\mathcal D$.
Consequently,
$$
2d_{\mathcal B}(x) + d_{\mathcal C}(x) + d_{\mathcal D}(x) \leq 2\ell + 2(n - 1) - 2\ell = 2(n - 1).
$$
By the induction hypothesis,
$$
2|\mathcal B(\overline x)| + |\mathcal C(\overline x)| + |\mathcal D(\overline x)| \leq 4\left\lfloor (n - 1)^2/4\right\rfloor.
$$
Thus
$$
2|\mathcal B| + |\mathcal C| + |\mathcal D| \leq 4 \left\lfloor (n - 1)^2/4\right\rfloor + 2(n - 1) \leq 4 \left\lfloor n^2/4\right\rfloor. \qquad \qedhere
$$
\end{proof}

Finally let us deduce \eqref{eq:1.3} from \eqref{eq:2.6}.
As noted after the statement, it is sufficient to prove \eqref{eq:1.3} for $t = 4$.

Let $\mathcal A, \mathcal B, \mathcal C, \mathcal D \subset {V\choose 2}$ be RBT-free.

\setcounter{observation}{7}
\begin{observation}
$\mathcal A \cap \mathcal B$, $\mathcal A \cup \mathcal B$, $\mathcal C$, $\mathcal D$ are RBT-free as well and
$|\mathcal A \cap \mathcal B| + |\mathcal A \cup \mathcal B| + |\mathcal C| + |\mathcal D| = |\mathcal A| + |\mathcal B| + |\mathcal C| + |\mathcal D|$.\hfill $\square$
\end{observation}

In view of the observation if $\mathcal A \not\subset \mathcal B$ and $\mathcal B \not\subset \mathcal A$, then we may replace $\mathcal A$ and $\mathcal B$ by $\mathcal A \cap \mathcal B$ and $\mathcal A \cup \mathcal B$.
Repeating this procedure after renaming the new families $\mathcal A, \mathcal B, \mathcal C, \mathcal D$ in some order, eventually we arrive at $\mathcal A, \mathcal B, \mathcal C, \mathcal D \subset {V\choose 2}$ that are RBT-free and nested, that is, satisfy $\mathcal A \subset \mathcal B \subset \mathcal C \subset \mathcal D$ as well.
Now by the RBT-free property $\mathcal B$ must be triangle-free.
Applying \eqref{eq:2.6} we infer
$|\mathcal A| + |\mathcal B| + |\mathcal C| + |\mathcal D| \leq 2|\mathcal B| + |\mathcal C| + |\mathcal D| \leq 4\left\lfloor n^2/4\right\rfloor$ as desired.\hfill $\square$

\section{The proof of Theorem \ref{th:1.5}}
\label{sec:3}

To avoid double indices let us rename the three families, $\mathcal B:= \mathcal G_1$, $\mathcal C := \mathcal G_2$, $\mathcal D := \mathcal G_3$.
Since $\mathcal B, \mathcal C, \mathcal D$ are RBT-free, $\mathcal B$ is triangle-free.
Set $\ell = \nu(\mathcal B)$.

If $\mathcal B$ is nearly matchable, i.e., $n \leq 2\ell + 2$, then $|\mathcal B||\mathcal C||\mathcal D| \leq \left\lfloor n^2/4\right\rfloor^3$ follows from Proposition \ref{prop:2.6}.

From now on we assume $n > 2\ell + 2$ and apply Proposition \ref{prop:1.2} to $\mathcal B$.
Let $V = X \cup Y \cup Z$ be the corresponding partition and $\bigl\{(x_i, y_i) : 1 \leq i \leq \ell\bigr\}$ a maximal matching in~$\mathcal B$.

Let $p$ be the maximum degree inside $X$ in the bipartite graph $\mathcal B \cap (X \times Z)$.
For convenience suppose $x_1$ has degree $p$ and $(x_1, z_j) : 1 \leq j \leq p$ are the edges from $x_1$ to $Z$.

The important thing to note is that being RBT-free implies that for $1 \leq j < j' \leq p$, $(z_j, z_{j'}) \notin \mathcal C \cup \mathcal D$.
Indeed, otherwise $\bigl\{x_1, z_j, z_{j'}\bigr\}$ would span a rainbow triangle (note that we use $\mathcal B \subset \mathcal C$, $\mathcal B \subset \mathcal D$ but do not need $\mathcal C \subset \mathcal D$ for this).

Let us provide upper bounds on $|\mathcal B|$, $|\mathcal C|$ and $|\mathcal D|$.
Set $q = |Z|$.

\begin{proposition} \rule{5mm}{0pt}
\label{prop:3.1}

\noindent
{\rm (i)} \quad\ $|\mathcal B| \leq \ell^2 + \ell p$.

\noindent
{\rm (ii)} \quad $\dfrac12\bigl(|\mathcal C| + |\mathcal D|\bigr) \leq \ell^2 + \ell q + {q\choose 2} - {p\choose 2} \leq \ell^2 + \ell p + \dfrac{q^2}{2} - \dfrac{p^2}{2}$.
\end{proposition}

\begin{proof}
(i) follows from Proposition \ref{prop:1.2} and the definition of $p$.
To prove (ii) note that the RBT-free property implies that for all $1 \leq i < i' \leq \ell$ out of the four edges joining $(x_i, y_i)$ and $(x_{i'}, y_{i'})$, counting with multiplicity, there are at most four in $\mathcal C \cup \mathcal D$.
This implies
$
\left|\mathcal C \cap {X \cup Y\choose 2}\right| + \left|\mathcal D \cap {X \cup Y\choose 2}\right| \leq 2\ell^2.
$
Then Lemma \ref{lem:2.4} yields
$ 
\left|\mathcal C \cap (X \cup Y)\times Z\right| + \left|\mathcal D \cap (X \cup Y)\times Z\right| \leq 2\ell q.
$ 
Finally from the above observation, at least ${p\choose 2}$ edges are missing from ${Z \choose 2}$ in both $\mathcal C$ and $\mathcal D$.
Thus
$$
\left|\mathcal C \cap {Z\choose 2}\right| + \left|\mathcal D \cap {Z\choose 2}\right| \leq 2 {q\choose 2} - 2{p\choose 2} = q^2 - p^2 - (q - p).
$$
Summing these inequalities yields (ii).
\end{proof}

In view of Proposition \ref{prop:3.1} in order to prove Theorem \ref{th:1.5} we should show
\beq
\label{eq:3.1}
(\ell^2 + \ell p)\left(\ell^2 + \ell q + \frac{q^2}{2} - \frac{p^2}{2}\right)^2 \leq \left\lfloor \frac{(2\ell + q)^2}{4}\right\rfloor^3.
\eeq

To avoid meticulous calculation we shall only deal in detail with the case $q$ is even, i.e., when we can remove the integer part symbol $\lfloor \ \ \rfloor$.
However it will be clear from the proof that for $q > 0$ the inequality is always strict and for $q \geq 3$ there is plenty of room left to take care of the difference of $1/4$ for $q$ odd.

For notational purposes define $0 \leq \alpha \leq \beta$ by $q = 2\ell\beta$, $p = 2\ell \alpha$.
Now \eqref{eq:3.1} is equivalent to
$$
(1 + 2\alpha) \bigl(1 + 2\beta + 2\beta^2 - 2\alpha^2\bigr)^2 \leq (1 + \beta)^6.
$$
Noting that $(1 + \alpha)^2 \geq 1 + 2\alpha$, it is sufficient to show
\beq
\label{eq:3.2}
(1 + \alpha) \bigl(1 + 2\beta + 2\beta^2 - 2\alpha^2\bigr) \leq (1 + \beta)^3.
\eeq
After expanding we get
$$
\alpha + (2\alpha \beta - 2\alpha^2) + \bigl(2\alpha(\beta - \alpha)(\beta + \alpha)\bigr) \leq \beta + \beta^2 + \beta^3.
$$
Now, $\alpha \leq \beta$ and $2\alpha \beta - 2\alpha^2 = 2(\beta - \alpha)\alpha < \beta^2/2$ and further
$2\alpha(\beta - \alpha) (\beta + \alpha) \leq 2 \cdot \frac{\beta^2}{4} \cdot 2\beta = \beta^3$.
Thus \eqref{eq:3.2} and thereby the theorem is proved.\hfill $\square$

\medskip
Let us mention that in fact we proved
$$
|\mathcal B||\mathcal C||\mathcal D| \leq \left(\frac{(2\ell + q)^2}{4}\right)^3 \biggm/\left(\frac{(1 + \alpha)^2}{1 + 2\alpha}\right)^2.
$$
Unless $\alpha$ is very-very small the RHS is smaller than the RHS of \eqref{eq:3.1} even for $q$ odd.
On the other hand if $\alpha$ is small, then \eqref{eq:3.1} can be easily proved.

Let us conclude this paper with the obvious conjecture.

\medskip
\noindent
{\bf Conjecture 3.}
{\it Suppose that $\mathcal B, \mathcal C, \mathcal D \subset {V\choose 2}$, $|V| = n$ and $\mathcal B, \mathcal C, \mathcal D$ are RBT-free.
Then }
\beq
\label{eq:3.3}
|\mathcal B||\mathcal C||\mathcal D| \leq \left\lfloor n^2/4\right\rfloor^3.
\eeq

\end{document}